\documentclass[fleqn,a4paper,12pt]{article}
\usepackage{amsmath,amsfonts,calrsfs,amssymb,color}
\usepackage{amsthm}

\newtheorem{Theorem}{Theorem}[section]
\newtheorem{Definition}[Theorem]{Definition}
\newtheorem{Proposition}[Theorem]{Proposition}
\newtheorem{Lemma}[Theorem]{Lemma}

\newtheorem{Remark}[Theorem]{Remark}

\makeatletter
\@addtoreset{equation}{section}

\makeatother

\def\R{\mathbb R}
\def\N{\mathbb N}

\def\ds{\displaystyle}

\title{\bf Uniqueness for continuity equations in Hilbert spaces with weakly differentiable drift}
\author{
Giuseppe Da Prato\\\normalsize Scuola Normale Superiore Pisa
\and
Franco Flandoli\\
\normalsize University of Pisa
\and
Michael R\"ockner\\
\normalsize University of Bielefeld
}
\date{ }
\begin{document}
\maketitle

\begin{abstract} 
 We prove uniqueness for continuity equations in Hilbert spaces $H$. The corresponding drift $F$ is assumed to be in a first order Sobolev space with respect to some Gaussian measure. As in previous work on the subject, the proof is based  on commutator estimates which are infinite dimensional analogues to the classical ones due to DiPerna--Lions. Our general approach is, however, quite different since, instead of considering renormalized solutions, we prove a dense range condition implying uniqueness. In addition, compared to known results by Ambrosio--Figalli and Fang--Luo, we use a different approximation procedure, based on a more regularizing Ornstein--Uhlenbeck semigroup and consider Sobolev spaces of vector fields taking values in $H$ rather than the Cameron--Martin space of the Gaussian measure. This leads to different conditions  on the derivative of $F$, which are incompatible with previous work on the subject. Furthermore, we can drop the usual exponential integrability conditions on the Gaussian divergence of $F$, thus improving known uniqueness results in this respect. 
\end{abstract}
\bigskip

\noindent {\bf 2000 Mathematics Subject Classification AMS}: 35F05, 58D20, 60H07  

  \bigskip

 \section{Introduction}
 Let $H$ be a separable real Hilbert space with inner product
 $\langle \cdot, \cdot \rangle$ and norm $|\cdot|$. Let $F:[0,\infty)\times H\to H$ be Borel measurable. In this paper we want to give a new proof for uniqueness of solutions to the corresponding continuity equations  informally given as
 \begin{equation}
 \label{e1.1}
 \frac{d}{dt}\;\mu_t+\mbox{\rm div}\,(F(t,\cdot)\mu_t)=0,\quad \mu_0=\zeta,
 \end{equation} 
where $\zeta$ is a given initial datum in $\mathcal P(H)$, i.e. a probability measure on the Borel $\sigma$-algebra $\mathcal B(H)$ of $H$, and the solution $t\mapsto\mu_t$ a curve in $\mathcal P(H)$. The divergence in \eqref{e1.1} is meant in the sense of distributions, more precisely one uses the duality between $\mathcal P(H)$ and a space of test functions on $[0,\infty)\times H$ which we denote by $\mathcal D_T$ and which will be specified below.
Then one can give  \eqref{e1.1} a rigorous meaning by a weak formulation. More precisely, we fix an orthonormal basis  $\{e_n:\;n\in\N\}$  of $H$, $T>0,$ and set $H_T:=[0,T]\times H$. Then we define $\mathcal D_T$ to be the linear space of all functions $u:H_T\to \R$ such that there exists $N\in\N$ such that
$$
u(t,x)=u_N(t,\langle e_1,x \rangle,\ldots,\langle e_n,x \rangle) ,\quad x\in H,
$$
for some $u_N\in C^1_b([0,T]\times \R^N)$ such that $u_N(T)=0$. Then \eqref{e1.1} can be rigorously written as
\begin{equation}
 \label{e1.2}
 \int_0^T\int_H \mathcal K_F\,u(s,x)\,\mu_s(dx)\,ds=-\int_Hu(0,x)\,\zeta(dx),\quad \forall\;u\in \mathcal D_T, 
 \end{equation} 
where for $(t,x)\in H_T$, $\mathcal K_F$ is a (degenerate) Kolmogorov operator defined by
\begin{equation}
 \label{e1.3}
   \mathcal K_Fu(t,x)=\frac{\partial}{\partial t}\;u(t,x)+\langle F(t,x),Du(t,x) \rangle 
 \end{equation}  
  and $Du(t,x)\in H$ is defined  through
 $$
    \langle Du(t,x),y   \rangle=u'(t,x)(y),\quad y\in H,
  $$
  where $u'(t,x)(\cdot)$ means first Fr\'echet derivative  of $u(t,\cdot) $ with respect to $x\in H$. We note that $ \mathcal D_T$ depends on the chosen orthonormal basis. But this is irrelevant because what is important about the chosen test functions spaces in regard to uniqueness, is that \eqref{e1.2} makes sense and that it is as small as possible (to make the uniqueness result as strong as possible). A minimal requirement is that it should separate the points  of $H$, which obviously holds for $ \mathcal D_T$  defined above by the Hahn--Banach theorem, which in turn by a monotone class argument implies that  $ \mathcal D_T$ is dense in every $L^p(H_T, \mbox{\boldmath $\nu$})$, $p\in [1,\infty),$ for any finite (nonnegative) measure $\mbox{\boldmath $\nu$}$ in $H_T$.\medskip
  
  The main aim of this paper is to find conditions on $F$ such that \eqref{e1.2} has at most one solution  for a given initial  condition
  $\zeta\in\mathcal P(H)$.
  
  In contrast to the Fokker--Planck equation where $\mathcal K_F$ in \eqref{e1.2}-\eqref{e1.3} has a second order part (in $x$), and uniqueness is known even for just measurable $F$ (satisfying some integrability assumption), provided the second order part is non degenerate (cf. \cite{BDPR11} and  the preprint \cite{BDPRS}), for the continuity equation at least weak differentiability of $F$ (or being of bounded variation) is required to hope to have uniqueness of solutions, even in finite dimensions (see \cite{DPL}, \cite{Am}). However, in order to define weak differentiability of a function one needs a reference measure, and  if $H=\R^n$, e.g. the Lebesgue measure is a natural choice. If $H$ is infinite dimensional,   Lebesgue measure does not exist and we have to choose  a reference measure on $H$. There is really no canonical choice, but a ``good'' choice is to take a non degenerate centered Gaussian measure $\mu$ on $H$, because the concept 
of weak differentiability (with respect to such a $\mu$) has been extensively developed in the past in the framework of  the Malliavin calculus (\cite{Malliavin}, \cite{Nualart}, \cite{Bo98}). This choice of a reference measure  was proposed in \cite{AF09}  and they proved existence and uniqueness of solutions to \eqref{e1.2} under certain conditions on the weak derivative and exponential $\mu$-integrability conditions on its $\mu$-divergence. (see \cite[Theorem 3.1]{AF09},  see also \cite{FL10} for improvements of the results on the corresponding transport equation in \cite{AF09}.)

In this paper, also taking a Gaussian measure $\mu$ as a reference measure, we prove uniqueness  for \eqref{e1.2} by  a completely different method. On the other hand, our assumption on the weak derivative  of $F$ is different and, in fact incompatible with that in \cite{AF09}, since we use $H$  instead of the Cameron--Martin space as tangent space when defining Sobolev spaces (see  Remark \ref{r2.5} below). As a consequence, in contrast to \cite{AF09} we do not need to assume any exponential $\mu$-integrability conditions on the Gaussian divergence of  $F$. The idea of proof is inspired by the uniqueness proof for Fokker--Planck equations in Hilbert spaces from \cite{BDPR10}, \cite{BDPR11}. More precisely, we prove a suitable rank condition for the Kolmogorov operator in \eqref{e1.3}. But to implement this idea we have to regularize with  a much more smoothing Ornstein--Uhlenbeck semi-group than the one in \cite{AF09},\cite{FL10}  (see Section 2 below). Crucial is again the commutator estimate, which as turns out, can be proved  also for this regularization (see Section 3).

Let us remark that here we use the commutator estimate to prove a range
condition, opposite to the classical works where the commutator estimate is
used to prove renormalization of weak solutions. It is at this point that, in \cite{AF09} and  \cite{FL10}, exponential integrability is necessary; for our range
condition we do not need it. Concerning the problem  of proving a range
condition itself, this is usually done by means of gradient estimates on
solutions, which is a difficult problem; here we have the gradient estimate
for free, see \eqref{e2.6a}, because it holds for the $P_{\epsilon}$-regularized
solution.

Choosing  a  reference measure as in \cite{AF09}, \cite{FL10} we also have to restrict to a sub-class of solutions $\mu_t,\;t\in[0,\infty),$   to \eqref{e1.2}, namely those satisfying 
\begin{equation}
 \label{e1.4}
  \mu_t(dx)dt=\rho(t,x)\mu(dx)dt,
 \end{equation} 
  for some functions  $\rho\in L^p(H_T,dt\otimes \mu)$, $p>1,$  and prove uniqueness in this class.
  
  It is the subject of our further study to relax this condition \eqref{e1.4}, e.g. by considering more general reference measures than Gaussian measures. First steps in this direction have recently be done in \cite{KR12}, where the Gaussian measure $\mu$ is replaced by a measure $\nu$ which is differentiable in the sense of Fomin (see \cite{Bo98}). In particular, one can take certain Gibbs measures for $\nu$. However, the techniques in that paper are entirely different from our approach here.\bigskip

  We end this section recalling some results about the Ornstein--Uhlenbeck semigroup
  $P_t$ needed in what follows.  First we    choose and fix an arbitrary centered, non degenerate, Gaussian measure $\mu$ on $H.$ Let $Q$ be its covariance operator. So, $Q$ is symmetric, nonnegative definite with kernel $=\{0\}$ and Tr $Q<\infty$. We also use the notation $\mu=N_Q$. Then
  $P_t$ is, for $\varphi\in B_b(H)$, defined as
    \begin{equation}
 \label{e1.5}
P_t\varphi(x)=\int_H\varphi(y) N_{T_tx,Q_t}(dy),\quad x\in H,
 \end{equation} 
 where
  \begin{equation}
 \label{e1.6}
 T_t:=e^{-\frac{t}2 Q^{-1}},\quad Q_t=QS^2_t,\quad S_t:=(1-T_t^2)^{1/2}.
 \end{equation} 
   $N_{T_tx,Q_t}$ denotes the Gaussian measure on $H$ with covariance operator $Q_t$ and mean $T_tx$ and
 $B_b(H)$ is the space of all real and bounded Borel funcions on $H$.
We note for further use that
 \begin{equation}
 \label{e1.7}
  T^2_\epsilon+ S^2_\epsilon=1.
 \end{equation} 
 Consequently the matrix on $H\times H$  
\begin{equation}
 \label{e1.8}
\mathcal R:= \begin{pmatrix}
 T_\epsilon& S_\epsilon\\
 -S_\epsilon&T_\epsilon
 \end{pmatrix},
  \end{equation} 
 is orthogonal, so that $\mathcal R$ is invariant for the measure $\mu\times\mu$ on $H\times H$.
 
 Since $N_{T_tx,Q_t}<\!\!<N_Q$, we can write
    \begin{equation}
 \label{e1.9}
P_t\varphi(x)=\int_H\varphi(y)\rho(t,x,y)\mu(dy),
 \end{equation} 
where
  \begin{equation}
 \label{e1.10}
 \begin{array}{l}
\rho(t,x,y)\\
\\
=K(t)\exp\{-\tfrac12\langle Q_t^{-1} T_tx, T_tx\rangle+\langle Q_t^{-1}  T_tx,y\rangle-\tfrac12(\langle Q_t^{-1} T_ty, T_ty\rangle\},
\end{array}
 \end{equation} 
where $K(t)=[\det(1-T_t^2)]^{-1/2}$.

  We notice, for further use, the following identities.
    \begin{equation}
 \label{e1.11}
D_x\rho(t,x,y)=Q_t^{-1}T_t (y-T_tx)
 \end{equation} 
   \begin{equation}
 \label{e1.12}
D_y\rho(t,x,y)=Q_t^{-1}T_t (x-T_ty)
 \end{equation} 
We finally recall the Mehler formula
  \begin{equation}
 \label{e1.13}
P_t\varphi(x)=\int_H\varphi(T_tx+S_ty)\mu(dy).
 \end{equation}

 \section{The main result and scheme of the proof}
 
 \begin{Definition} 
 \label{d2.1}
 A family $(\mu_t)_{t\in[0,T]}$ is called  a solution of the (heuristic) continuity equation \eqref{e1.1} if $\mu_t\in\mathcal P(H)$ for every $t\in[0,T]$,   $t\mapsto \mu_t(A)$ is $\mathcal  B(H)$-measurable for all $A\in\mathcal  B(H)$,    $F\in L^1(H_T,\mu_tdt)$
 and \eqref{e1.2}  holds.
 \end{Definition}

As mentioned in the introduction we need a reference measure on $H$. So let $\mu=N_Q$ be the   centred, non degenerate, Gaussian measure on $H$ from the Introduction with covariance operator  $Q$.   Let $\{e_k:\;k\in\N\}$ be the eigenbasis of $Q$ and $\lambda_k\in(0,\infty)$ the corresponding  eigenvalues (i.e. $Qe_k=\lambda_ke_k,\;k\in\N$) numbered in decreasing order. Let the test function space $\mathcal D_T$ be defined as in the introduction with respect to this
 orthonormal basis $\{e_k:\;k\in\N\}$. 
 
 Define the following subclass $\mathcal M_ {F,\zeta,p}$ of solutions to \eqref{e1.1} for fixed initial condition $\zeta\in  \mathcal  P(H)$ and fixed $p\in [1,\infty]$. $\mathcal M_ {F,\zeta,p}$ is defined to be the set of all measures $ \mbox{\boldmath$\mu$}(dt,dx)=\mu_t(dx)dt$  such that  $(\mu_t)_{t\in[0,T]}$ is   a solution to \eqref{e1.1} in the sense of Definition \ref{d2.1} which satisfy
   \begin{equation}
 \label{e2.1}
 \mu_t(dx)dt=\rho(t,x)\mu(dx)dt,\quad\mbox{\rm for some}\; \rho\in L^p(H_T,dt\otimes \mu).
 \end{equation} 
 Clearly,  $\mathcal M_ {F,\zeta,p}$ is a convex set.

The following result is inspired by \cite{BDPR10},  \cite{BDPR11}
\begin{Proposition}
\label{p2.2}
Suppose the following rank condition holds:
$$
 \mathcal K_F (\mathcal D_T)\;\mbox{\it is dense in}\;L^{p'}(H_T,dt\otimes \mu), \eqno{(\mathcal R)}
$$
where $p\in [1,\infty]$ and $p'=\tfrac{p}{p-1}$. Then $\mathcal M_ {F,\zeta,p}$ contains at most one element.
\end{Proposition} 
\begin{proof}
Let $ \mu^{(i)}_t(dx)dt=\rho^{(i)}(t,x)\mu(dx)dt,$ $i=1,2$ be two elements in $\mathcal M_ {F,\zeta,p}$. Then by \eqref{e1.2}
$$
\int_0^T\int_H \mathcal K_F u(t,x)(\rho^{(2)}(t,x)-\rho^{(1)}(t,x))\mu(dx)dt=0,\quad\forall\;u\in \mathcal D_T.
$$
Hence $(\mathcal R)$ implies  $\rho^{(1)}=\rho^{(2)}$.

\end{proof}
Let us briefly recall the notion of (some) Sobolev spaces of functions on $H$ with respect to $\mu$.

Let $\mathcal F C_b^1$ (``finitely based $C_b^1$ functions'') denote the linear space of all functions   $\varphi:H\to \R$  such that for some $N\in\N$
$$
\varphi(x)=\varphi_N(\langle e_1,x \rangle,....,\langle e_N,x \rangle),\quad x\in H, 
$$
for some $\varphi_N\in C_b^1(\R^N)$. For $p\in[1,\infty)$ equip
$\mathcal F C_b^1$ with the norm
$$
\|\varphi\|_{1,p}:= \left(\int_H(|D\varphi(x)|^p+ |\varphi(x)|^p)\mu(dx)\right)^{1/p},
$$
where $D\varphi(x)$ is the unique element in $H$ such that
\begin{equation}
\label{e2.1'}
\langle D\varphi(x),y \rangle_H=\varphi'(x)(y)=\frac{\partial\varphi}{\partial y}(x), \quad y\in H, 
\end{equation}
where $\frac{\partial\varphi}{\partial y}$ means partial derivative in the direction $y$ and  $\varphi'$ denotes the Fr\'echet derivative of $\varphi$.
Then it is well-known (see e.g. \cite{DP11}) that  $\|\varphi\|_{1,p}$ is closable over $L^p(H,\mu)$ so that
$$
W^{1,p}(H,\mu)=\overline{\mathcal F C_b^1}^{\|\cdot\|_{1,p}} \;(=\mbox{\rm completion of $\mathcal F C_b^1$ with respect to $\|\cdot\|_{1,p}$})
$$
is a subspace of  $L^p(H,\mu)$. Likewise as this Sobolev space
of functions one   defines Sobolev spaces of vector fields $F:H\to H$ and even of time dependent vector fields $F:H_T\to H$ as follows: let $\mathcal V\!\mathcal F C_{b,T}^1$ (``finitely based $C_b^1$ vector fields'') denote the linear space of all maps $F:H_T\to H$  such that for some $N\in\N$
$$
F(t,x)=\sum_{i=1}^N g_i(t,x)e_i,\quad (t,x)\in H_T,
$$
for some $g_i:H_T\to \R$ of type
$$
g_i(t,x)=g_{i,N}(t, \langle e_1,x \rangle,....,\langle e_N,x \rangle),\quad x\in H, 
$$
with $g_{i,N}\in C_{b,T}^1([0,T]\times \R^N)$. For $p\in[1,\infty)$ we equip
$\mathcal V\!\mathcal F C_{b,T}^1$ with the norm
$$
\|F\|_{1,p,T}:=\left(\int_0^T  \int_H(\|DF(t,x)\|_{L_2(H,H)}^p+|F(t,x)|_{H}^p)\mu(dx)dt\right)^{1/p} 
$$
where $L_2(H,H)$ denotes the linear space of all Hilbert--Schmidt
operators on $H$ with corresponding norm $\|\cdot\|_{L_2(H,H)}$ and
$$
DF(t,x):=\sum_{i=1}^N \langle Dg_i(t,x),\cdot\rangle e_i\in L_2(H,H).
$$
Again, it is well known that this norm is closable in $L^p(H_T;L_2(H,H),\mu))$. Hence we can define the Sobolev space of time dependent vector fields by
$$
L^p([0,T];W^{1,p}(H;H,\mu))=\mbox{\rm completion of}\;\mathcal V\!\mathcal F C_{b,T}^1\;\mbox{\rm with respect to}\; \|\cdot\|_{1,p,T},
$$
 which by  closability is a subspace in $L^p(H_T;H,\mu)$.\medskip

Now we can formulate our main result. 
\begin{Theorem}
\label{t2.3}
Let $p\in(2,\infty)$ and suppose that, for some $s>p'=\frac{p}{p-1}$, we have   $F\in L^{s}([0,T];W^{1,s}(H;H,\mu))$ and that, in addition,
\begin{equation}
 \label{e2.2}
 F(H_T)\subset Q^{1/2}(H),\;\mbox{\it and}\;\int_0^T\int_H|Q^{-1/2}F(t,x)|^{s}\mu(dx)dt<\infty.
 \end{equation} 
Then the rank condition $(\mathcal R)$ holds, hence by Proposition \ref{p2.2}  $\mathcal  M_{ F,\zeta, p}$ contains at most one element.
\end{Theorem} \medskip

 The rest of this section is devoted to reduce  the proof of  $(\mathcal R)$ and hence of Theorem \ref{t2.3} to Proposition \ref{p2.4} below, which is a commutator estimate for a suitable  regularization through the Mehler type semigroup $P_t,\;t\ge 0,$ of integral operators on $B_b(H)$ defined in \eqref{e1.5} (see also  \eqref{e1.13})
  Let us define the commutator for  $u\in \mathcal D_T$, $F\in  \mathcal  V\mathcal F C^1_{b,T},$ $(t,x)\in H_T$
 \begin{equation}
 \label{e2.4}
 B_\epsilon(u,F)(t,x):=\langle F(t,x),DP_\epsilon(u(t,\cdot))(x) \rangle-P_\epsilon(\langle F(t,\cdot),Du(t,\cdot) \rangle)(x).
 \end{equation} 
 \begin{Proposition}
\label{p2.4}
Let $p\in(2,\infty)$ and $r\in[1,\infty)$, $s\in(1,2]$ such that $\frac1{p'}=\frac1r+\frac1{s}.$
Then:
\begin{enumerate}
\item[(i)]   There exists $C\in(0,\infty)$ such that
$$
\begin{array}{l}
\ds\left(\int_0^T\int_H|B_\epsilon(u,F)|^{p'}d\mu\,dt   \right)^{1/p'}\\
\\
\ds\le C\|u\|_{L^r(H_T,dt\otimes \mu)}\;\,\left( \|F\|_{1,s,T}+\|Q^{-1/2}F\|_{L^s(H_T,dt\otimes\mu)}   \right),
\end{array}
$$
for all $ u\in\mathcal D_T,\;F\in  \mathcal V\mathcal FC^1_{b,T}.$ In particular, $B_\epsilon$ extends to a continuous bilinear map (denoted by the same symbol)
$$
\begin{array}{l}
B_\epsilon: L^r(H_T,dt\otimes \mu)\times L^s([0,T];W^{1,s}(H;H,\mu)\cap L^s(H; Q^{1/2}H;\mu))\\
\\
\to L^{p'}(H_T,dt\otimes \mu).
\end{array}
$$

\item[(ii)] $B_\epsilon(u,F)\to 0$ in  $L^{p'}(H_T,dt\otimes \mu)$ as
$\epsilon\to 0$, for all 
$$(u,F)\in L^r(H_T,dt\otimes \mu)\times L^s([0,T];W^{1,s}(H;H,\mu)\cap L^s(H;Q^{1/2}H,\mu)).$$

\end{enumerate}
\end{Proposition}

The proof of Proposition \ref{p2.4}-(i) is carried out in Section 3 below. Assertion (ii) obviously holds for all $u\in \mathcal D_T$, $F\in\mathcal V\mathcal FC^1_{b,T}$. But by (i), $B_\epsilon,$ $\epsilon\in[0,1],$ are equicontinuous on
$$
L^r(H_T,dt\otimes \mu)\times L^s([0,T];W^{1,s}(H;H,\mu)\cap L^s(H;Q^{1/2}H,\mu)),
$$
which contains $\mathcal D_T\times \mathcal V\mathcal FC^1_{b,T}$ as a dense set. Hence (ii) follows.

Let us now show that Proposition \ref{p2.4} implies Theorem \ref{t2.3}.\medskip

{\bf Claim}. Proposition \ref{p2.4} implies $(\mathcal R)$.

\begin{proof}
Let $f\in \mathcal D_T$ and $r,s\in[1,\infty)$ be as in Proposition \ref{p2.4}  such that $s\in(p',2)$. By definition of $L^s([0,T];W^{1,s}(H;H,\mu))$ there  exists $F_n\in \mathcal V\!\mathcal F C_{b,T}^1$, $n\in\N$, converging to to $F$ w.r.t.  $\|\cdot\|_{1,s,T}$ and in the sense of Lemma \ref{l4.1a} of Appendix A. Since $F_n$ is smooth and finitely based, there exists a solution $u_n\in \mathcal D_T$ of
\begin{equation}
 \label{e2.5}
 \left\{\begin{array}{l}
 \ds\frac{\partial u_n}{\partial t}+\langle F_n,Du_n \rangle =f,\\
 \\
 u_n(T,\cdot)=0
 \end{array}\right.
 \end{equation} 
 We namely set
 $$
 u_n(t,x)=\int_t^T f(s-t,\xi_n(s,t,x))ds
 $$
 where the characteristics $\xi_n(s,t,x)$ are, as well known,  the solution to
 $$
 \frac{\partial }{\partial s}\xi_n(s,t,x)=F_n(s,\xi_n(s,t,x)),\quad \xi_n(t,t,x)=x.
 $$
 
 Applying $P_\epsilon$ for $\epsilon>0$ to \eqref{e2.5} we obtain
 \begin{equation}
 \label{e2.6}
 \frac{\partial P_\epsilon u_n}{\partial t}+\langle F,DP_\epsilon u_n \rangle =P_\epsilon f+\langle F-F_n,DP_\epsilon u_n \rangle+B_\epsilon(F_n,u_n),
 \end{equation} 
 note that $P_\epsilon u_n\in \mathcal D_T$.
 
By the maximum  principle we have
$$
\|u_n\|_\infty\le \|f\|_\infty,\quad\forall\;n\in \N,
$$
and by well known smoothing properties of $P_\epsilon$ (see e. g. \cite{DP04}) we know that for some $C\in(0,\infty)$
 \begin{equation}
 \label{e2.6a}
\|DP_\epsilon u_n\|_\infty\le C\epsilon^{-1/2}\|u_n\|_\infty\le C\epsilon^{-1/2}\|f\|_\infty,\quad\forall\;n\in \N.
 \end{equation}
Hence, passing to a subsequence if necessay, we may assume that $u_n\to u$ in $L^r(H_T,dt\otimes \mu)$ weakly. But for every $v\in L^p(H_T,dt\otimes \mu)$ by Proposition \ref{p2.4}  
$$
\begin{array}{l}
\ds\left|\int_0^T\int_Hv(B_\epsilon(u_n,F_n) -B_\epsilon(u,F))d\mu\,dt \right|\\
\\
\ds=\left|\int_0^T\int_HvB_\epsilon(u_n,F_n-F)  d\mu\,dt+\int_0^T\int_HvB_\epsilon(u_n-u,F) d\mu\,dt  \right|\\
\\
\ds\le C\|v\|_{L^p(H_T,dt\otimes \mu)}\;\|u\|_{L^r(H_T,dt\otimes \mu)}\\
\\
\ds\times  \left[\|F_n-F\|_{1,s,T}+\|Q^{-1/2}(F_n-F)\|_{L^s(H_T,dt\otimes \mu)}\right]\\
\\
\ds+ \int_0^T\int_HB_\epsilon(\cdot,F)^* v\,(u_n-u)\, d\mu\,dt    \\
\\
\to 0\quad\mbox{as}\;n\to\infty,
\end{array} 
$$where $B_\epsilon(\cdot,F)^*\in L(L^p(H_T,dt\otimes \mu),L^{r'}(H_T,dt\otimes \mu)),$ $r'=\frac{r}{r-1},$ is the adjoint of the linear bounded operator in $L(L^r(H_T,dt\otimes \mu),L^{p'}(H_T,dt\otimes \mu))$ given by
$$
u\mapsto B_\epsilon(u,F).
$$
Here we have used Lemma \ref{l4.1a}.
Hence
$$
B_\epsilon(u_n,F_n)\to B_\epsilon(u,F)\quad\mbox{weakly in}\;L^{p'}(H_T,dt\otimes \mu)
$$
By Proposition \ref{p2.4}(ii), $B_\epsilon(u,F)\to 0$ in $L^{p'}(H_T,dt\otimes \mu)$  as $\epsilon\to 0$, hence also  with respect to the weak topology on $L^{p'}(H_T,dt\otimes \mu)$. Since the latter is metrizable on norm balls in $L^{p'}(H_T,dt\otimes \mu)$ and since $s\ge p'$, we see
that the right hand side of \eqref{e2.6}, weakly converges to $f$  in $L^{p'}(H_T,dt\otimes \mu)$ when we let first $n\to\infty$ and then  $\epsilon\to 0$. But obviously the left hand side of \eqref{e2.6} is in $\mathcal K_F(\mathcal D_T)$.  Therefore, we obtain that  $\mathcal K_F(\mathcal D_T)$ is weakly dense in $L^{p'}(H_T,dt\otimes \mu)$, since it cointains $\mathcal D_T$ as a dense subset. Hence $(\mathcal R)$ follows, since $\mathcal K_F(\mathcal D_T)$ is convex (even linear).
\end{proof}
\begin{Remark}
\label{r2.5}
\em

Let us compare our main result Theorem \ref{t2.3} with the corresponding result about uniqueness in \cite{AF09} (i.e. the uniqueness part of \cite[Theorem 3.1]{AF09}.)

We shall in fact see that they are incompatible. First of all, since we work on a separable Hilbert space $H$ and the authors of the above paper work on a  separable Banach space $E$, to compare we have to assume that $E$ is a separable Hilbert space. They consider also another Hilbert space  which is contained in $E=H$  and which can easily   be seen to be identical to $Q^{1/2}H=:\mathcal H$ with norm $|\cdot|_{\mathcal H}=|Q^{-1/2}\cdot|_H$. $\mathcal H$ is considered in \cite{AF09} as a tangent space at every point in $H$, while in our framework the tangent space to $H$ is $H$ itself.

 While condition \eqref{e2.2} is also assumed in \cite{AF09},  instead of our condition
\begin{equation}
\label{e2.7}
F\in L^s([0,T];W^{1,s}(H;H,\mu)),
\end{equation}
the authors assume that
\begin{equation}
\label{e2.8}
F\in L^s([0,T];W^{1,s}(H;\mathcal H,\mu)),
\end{equation}
which in turn is defined to be the completion of $\mathcal V\mathcal FC^1_{b,T}$ with respect to the norm
$$
\left(\int_0^T\int_H\left(\|D_\mathcal HF(t,x)\|^s_{L_2(\mathcal H,\mathcal H)}+|F(t,x)|^s_{\mathcal H}   \right)\mu(dx)\,dt     \right)^{1/s},  
$$
where $L_2(\mathcal H,\mathcal H)$ is the space of Hilbert--Schmidt operators from $\mathcal H$ to $\mathcal H$ and analogously to \eqref{e2.1'} for $\varphi\in \mathcal FC_b^1,\;x\in H,$ $D_\mathcal H\varphi(x)$ is the unique element in $\mathcal H$ such that
\begin{equation}
\label{e2.9}
 \langle D_\mathcal H\varphi(x),y   \rangle_\mathcal H=\varphi'(x)(y)=\frac{\partial\varphi}{\partial y}(x),\quad y\in \mathcal H.
\end{equation}
Correspondingly, for $F=\sum_{i=1}^Ng_i\;e_i\in \mathcal V\mathcal FC^1_{b,T}$, $(t,x)\in H_T$
$$
D_\mathcal HF(t,x):=\sum_{i=1}^N\langle D_\mathcal H g_i(t,x),\cdot\rangle_\mathcal H\;e_i\;(\in L_2(\mathcal H,\mathcal H)).
$$
Note that clearly $\widetilde{e_j}:=\lambda_j^{1/2}e_j,\;j\in\N,$ is an orthonormal basis in $\mathcal H$, hence
$$
\begin{array}{l}
\ds \|D_\mathcal H F(t,x) \|^2_{L_2(\mathcal H,\mathcal H)}=\sum_{j=1}^\infty|D_\mathcal H F(t,x)(\widetilde{e_j})|^2_\mathcal H\\
\\
\ds=\sum_{j=1}^\infty\sum_{i=1}^N\lambda_j\langle D_\mathcal H g_i(t,x),e_j\rangle^2_\mathcal H \langle e_i,e_i   \rangle_\mathcal H\stackrel{\eqref{e2.9}}{=}\sum_{i,j=1}^N\frac{\lambda_j}{\lambda_i}
\left( \frac{\partial g_i}{\partial e_j}(t,x)  \right)^2,  
\end{array} 
$$
whereas similarly
$$
\|D F(t,x)\|^2_{L^2(H,H)}=\sum_{i,j=1}^N \left( \frac{\partial g_i}{\partial e_j}(t,x)   \right) ^2.
$$
Therefore, the spaces in conditions \eqref{e2.7}, \eqref{e2.8} are incompatible and hence so are \eqref{e2.7}and \eqref{e2.8}. A further difference to \cite{AF09} is that unlike in (the uniqueness part of)
\cite[Theorem 3.1]{AF09} we do not have to assume any exponential $\mu\otimes dt$-integrability of the Gaussian divergence of $F$, i.e. of the negative part of $(-D^*_\mathcal H F)$ where $D^*$ is the adjoint of
$$
D:W^{1,2}_\mathcal H(H,\mu)\subset L^2(H,\mu)\to L^2(H;\mathcal H,\mu).
$$
It is easy to construct examples where this exponential integrability does not hold for $F$, while $F$ satisfies all other assumptions in Theorm \ref{t2.3}.
\end{Remark}

\section{Proof of Proposition \ref{p2.4}(i)}

 \subsection{A representation formula for the commutator}  
 We shall use the notation
  \begin{equation}
 \label{e3.1}
\mbox{\rm div}_QF(t,x):=\mbox{\rm Tr}\;[DF(t,x)]-\langle Q^{-1}x,F(t,x) \rangle.
\end{equation}
\begin{Proposition}
 \label{p3.1}
We have
   \begin{equation}
 \label{e3.2}
 \begin{array}{l}
\ds B_\epsilon (u,F)(t,x) =\int_H\mbox{\rm div}_Q F(t,T_\epsilon x+S_\epsilon y)\,u(t,T_\epsilon x+S_\epsilon y)\mu(dy) \\
\\
\ds-\int_H[g_\epsilon(t,T_\epsilon x+ S_\epsilon y, -S_\epsilon x+ T_\epsilon y)-g_\epsilon(t,x,y)]\,u(T_\epsilon x+ S_\epsilon y)\mu(dy)\\
\\
=:B^1_\epsilon (u,F)(t,x)+B^2_\epsilon (u,F)(t,x),
 \end{array}
 \end{equation}  
 where
  \begin{equation}
 \label{e3.3}
  g_\epsilon(t,x,y):=\langle \tfrac{Q^{-1} T_\epsilon}{S_\epsilon}\;F(t,x), y \rangle.
 \end{equation}  
 
 \end{Proposition}
 \begin{proof}
 Concerning the second addendum of the commutator \eqref{e2.4}, we have by \eqref{e1.9}, using a well known integration by parts formula for Gaussian measures,
 $$
 \begin{array}{l}
\ds (P_\epsilon(\langle F(t,\cdot),D_xu(t,\cdot) \rangle)(x)=\int_H\langle F(t,y),D_yu(t,y) \rangle\,\rho(\epsilon,x,y)\mu(dy)\\
\\
\ds=-\int_H\mbox{\rm div}\; F(t,y)\,u(t,y)\,\rho(\epsilon,x,y)\mu(dy)\\
\\
\ds-\int_H\langle F(t,y),D_y\rho(\epsilon,x,y) \rangle\,u(t,y)\mu(dy)
\\
\\
\ds +\int_H\langle Q^{-1}y,F(t,y) \rangle\,u(t, y)\,\rho(\epsilon,x,y)\mu(dy).
\end{array}
 $$
 Taking into account \eqref{e1.12}, yields
   \begin{equation}
 \label{e3.4}
 \begin{array}{l}
\ds (P_\epsilon(\langle F(t,\cdot),D_xu(t,\cdot) \rangle)(x) =-\int_H\mbox{\rm div}_Q F(t,y)\,u(t,y)\,\rho(\epsilon,x,y)\mu(dy)\\
\\
\ds-\int_H\langle F(t,y),Q_\epsilon^{-1}T_\epsilon (x-T_\epsilon y) \rangle\,u(t,y)\,\rho(\epsilon,x,y)\,\mu(dy).
\end{array}
 \end{equation} 
   Concerning the first addendum of the commutator, we have, taking into account \eqref{e1.11},
 $$
 \begin{array}{l}
\ds \langle F(t,x),D_xP_\epsilon u(t,x) \rangle=\int_H\langle F(t,x),D_x\rho(\epsilon,x,y) \rangle\,u(t,y)\,\mu(dy)\\
\\
\ds=\int_H\langle F(t,x),Q_\epsilon^{-1}T_\epsilon (y-T_\epsilon x) \rangle\,u(t,y)\,\rho(\epsilon,x,y)\,\mu(dy).
\end{array}
 $$
 So, we obtain
 \begin{equation}
 \label{e3.5}
 \begin{array}{l}
\ds B_\epsilon(u,F)(t,x) =\int_H\mbox{\rm div}_Q F(t,y)\,u(t,y)\,\rho(\epsilon,x,y)\mu(dy)\\
\\
\ds+\int_H\langle F(t,x),Q_\epsilon^{-1}T_\epsilon (y-T_\epsilon x) \rangle\,u(t,y)\,\rho(\epsilon,x,y)\,\mu(dy)\\
\\
\ds+\int_H\langle F(t,y),Q_\epsilon^{-1}T_\epsilon  (x-T_\epsilon y) \rangle\,u(t,y)\,\rho(\epsilon,x,y)\,\mu(dy).
\end{array}
 \end{equation} 
 Since $\rho(\epsilon,x,y)\,\mu(dy)=N_{T_\epsilon x,Q_\epsilon}(dy)$ we  can write \eqref{e3.5} as
$$
 \begin{array}{l}
\ds B_\epsilon(u,F)(t,x) =\int_H\mbox{\rm div}_Q F(t,T_\epsilon x+y)\,u(t,T_\epsilon x+y)\,N_{Q_\epsilon}(dy)\\
\\
\ds+\int_H\langle F(t,x),Q_\epsilon^{-1}T_\epsilon y \rangle\,u(t,T_\epsilon x+y)\,N_{Q_\epsilon}(dy)\\
\\
\ds+\int_H\langle F(t,T_\epsilon x+y),Q_\epsilon^{-1}T_\epsilon  (x-T_\epsilon (y+T_\epsilon x)) \rangle\,u(t,T_\epsilon x+y)\,N_{Q_\epsilon}(dy).
\end{array}
$$ 
Since  $x-T_\epsilon (y+T_\epsilon x)=x-T^2_\epsilon x-T_\epsilon  y$ $= S_\epsilon^2x-T_\epsilon  y$,  using the Mehler formula \eqref{e1.13} we have
$$
 \begin{array}{l}
\ds B_\epsilon(u,F)(t,x) =\int_H\mbox{\rm div}_Q F(t,T_\epsilon x+S_\epsilon y)\,u(t,T_\epsilon x+S_\epsilon y)\,N_{Q}(dy)\\
\\
\ds+\int_H\langle F(t,x),Q_\epsilon^{-1}T_\epsilon S_\epsilon  y \rangle\,u(t,T_\epsilon x+S_\epsilon y)\,N_{Q}(dy)\\
\\
\ds+\int_H\langle F(t,T_\epsilon x+S_\epsilon y),Q_\epsilon^{-1}T_\epsilon S_\epsilon(S_\epsilon x-T_\epsilon  y) \rangle\,u(t,T_\epsilon x+S_\epsilon y)\,N_{Q}(dy),
\end{array}
$$ 
which coincides with \eqref{e3.2}. 
 \end{proof}
 We write now $B^2_\epsilon (u,F)(t,x)$ in a more suitable form.
 \begin{Proposition}
 \label{p3.2}
 We have
  \begin{equation}
 \label{e3.6}
  \begin{array}{l}
  \ds B^2_\epsilon (u,F)(t,x)\\
  \\
  \ds = \frac\epsilon2\int_H\int_0^1 [\langle \tfrac{Q^{-1} T_\epsilon}{S_\epsilon}\;DF(t,x_\xi) (Q^{-1} \tfrac{T_{\epsilon\xi}}{S_{\epsilon\xi} }y_\xi),y_\xi\rangle+\mbox{\rm div}\;G(t,x_\epsilon)]u(t,x_1)\,d\xi\,\mu(dy)\\
  \\
\ds  + \frac\epsilon2\int_H\int_0^1\mbox{\rm div}_Q\;G(t,x_\xi)u(t,x_1)\,d\xi\,\mu(dy):=B^{2,1}_\epsilon (u,F)(t,x)+B^{2,2}_\epsilon (u,F)(t,x), 
   \end{array}
 \end{equation}
 where
  \begin{equation}
 \label{e3.7}
 G(t,x)=Q^{-1}\tfrac{T_\epsilon}{S_\epsilon}\tfrac{T_{\epsilon\xi}}{S_{\epsilon\xi} }F(t,x),
 \end{equation}
 
 \end{Proposition}
 \begin{proof}
 
 We start from the expression of  $B^2_\epsilon (u,F)(t,x)$ given by  \eqref{e3.2}
  and for any $\xi\in[0,1]$ we set
\begin{equation}
 \label{e3.8}
 x_\xi=T_{\epsilon\xi} x+ S_{\epsilon\xi} y,\quad y_\xi=-S_{\epsilon\xi} x+T_{\epsilon\xi} y,
 \end{equation} 
which implies
$$
 x=T_{\epsilon\xi} x_\xi- S_{\epsilon\xi} y_\xi,\quad y=S_{\epsilon\xi} x_\xi+T_{\epsilon\xi} y_\xi.
$$ 
Notice that
$$
\begin{array}{l}
T_{\epsilon\xi} x- \tfrac{T^2_{\epsilon\xi}}{S_{\epsilon\xi} }y= -\tfrac{T_{\epsilon\xi}}{S_{\epsilon\xi} }y_\xi,\quad
T_{\epsilon\xi} y- \tfrac{T^2_{\epsilon\xi}}{S_{\epsilon\xi} }x= -\tfrac{T_{\epsilon\xi}}{S_{\epsilon\xi} }x_\xi.
\end{array}
$$
Therefore we can write
 \begin{equation}
 \label{e3.9}
  B^2_\epsilon (u,F)(t,x)= \int_H[ g_\epsilon(x_1, y_1 )-g_\epsilon(t,x,y)]\,u(t,x_1)\mu(dy)
 \end{equation}
and,  taking into account that
 \begin{equation}
 \label{e3.10}
 D_\xi x_\xi=-\tfrac12\,Q^{-1}\epsilon (T_{\epsilon\xi} x- \tfrac{T^2_{\epsilon\xi}}{S_{\epsilon\xi} }y),\quad D_\xi y_\xi=\tfrac12\,Q^{-1}\epsilon(\tfrac{T^2_{\epsilon\xi}}{S_{\epsilon\xi} }x-T_{\epsilon\xi} y),
\end{equation}
we have
  \begin{equation}
 \label{e3.11}
 \begin{array}{l}
  \ds g_\epsilon(t,x_{1}, y_{1})-g_\epsilon(t,x,y)=\int_0^1 D_\xi g_\epsilon(t,x_{\xi}, y_{\xi})d\xi\\
  \\
  \ds=\int_0^1[D_xg_\epsilon(t,x_{\xi}, y_{\xi}) D_\xi x_\xi+D_yg_\epsilon(t,x_{\xi}, y_{\xi}) D_\xi y_\xi]\;d\xi\\\\
  \ds=\int_0^1D_xg_\epsilon(t,x_{\xi}, y_{\xi}) \left(-\tfrac12\,Q^{-1}\epsilon (T_{\epsilon\xi} x- \tfrac{T^2_{\epsilon\xi}}{S_{\epsilon\xi} }y)\right)d\xi\\
  \\
\ds +\int_0^1 D_yg_\epsilon(t,x_{\xi}, y_{\xi}) \left(\tfrac12\,Q^{-1}\epsilon(\tfrac{T^2_{\epsilon\xi}}{S_{\epsilon\xi} }x-T_{\epsilon\xi} y)\right)d\xi   
     \end{array}
 \end{equation}
Therefore
  \begin{equation}
 \label{e3.12}
  \begin{array}{l}
  \ds B^2_\epsilon(u, F)(t,x)\\
  \\
  \ds = \int_H\int_0^1D_xg_\epsilon(t,x_{\xi}, y_{\xi}) \left(-\tfrac12\,Q^{-1}\epsilon (T_{\epsilon\xi} x- \tfrac{T^2_{\epsilon\xi}}{S_{\epsilon\xi} }y)\right)\;u(t,x_1)\,d\xi\,\mu(dy)\\
  \\
\ds  + \int_H\int_0^1D_yg_\epsilon(t,x_{\xi}, y_{\xi}) \left(\tfrac12\,Q^{-1}\epsilon(\tfrac{T^2_{\epsilon\xi}}{S_{\epsilon\xi} }x-T_{\epsilon\xi} y)\right)\,d\xi\,\mu(dy)
   \end{array}
 \end{equation}
 But
 $$
 D_x g_\epsilon(t,x_\xi,y_\xi)z=\langle \tfrac{Q^{-1} T_\epsilon}{S_\epsilon}\;DF(t,x_\xi)z, y_\xi  \rangle,
 $$
 
 $$
 D_y g_\epsilon(t,x_\xi,y_\xi)z=\langle \tfrac{Q^{-1} T_\epsilon}{S_\epsilon}\;F(t,x_\xi), z \rangle.
 $$ 
 Therefore from \eqref{e3.12} we get
  \begin{equation}
 \label{e3.13}
  \begin{array}{l}
  \ds B^2_\epsilon(F, u)(t,x)\\
  \\
  \ds =\tfrac\epsilon2 \int_H\int_0^1 \left< \tfrac{Q^{-1} T_\epsilon}{S_\epsilon}\;DF(t,x_\xi) (- Q^{-1}  (T_{\epsilon\xi} x- \tfrac{T^2_{\epsilon\xi}}{S_{\epsilon\xi} }y)),y_\xi\right>\;u(t,x_1)\,d\xi\,\mu(dy)\\
  \\
\ds  +\tfrac\epsilon2 \int_H\int_0^1\left< \tfrac{Q^{-1} T_\epsilon}{S_\epsilon}\;F(t,x_\xi),  Q^{-1} (\tfrac{T^2_{\epsilon\xi}}{S_{\epsilon\xi} }x-T_{\epsilon\xi} y) \right> u(t,x_1)\,d\xi\,\mu(dy). 
   \end{array}
 \end{equation}
Equivalently
 \begin{equation}
 \label{e3.14}
  \begin{array}{l}
  \ds B^2_\epsilon (u,F)(t,x)\\
  \\
  \ds = \frac\epsilon2\int_H\int_0^1 \left< \tfrac{Q^{-1} T_\epsilon}{S_\epsilon}\;DF(t,x_\xi) (Q^{-1} \tfrac{T_{\epsilon\xi}}{S_{\epsilon\xi} }y_\epsilon),y_\xi\right>u(t,x_1)\,d\xi\,\mu(dy)\\
  \\
\ds  + \frac\epsilon2\int_H\int_0^1\left< \tfrac{Q^{-1} T_\epsilon}{S_\epsilon}\;F(t,x_\xi),Q^{-1} \tfrac{T_{\epsilon\xi}}{S_{\epsilon\xi} }x_\xi\right> u(t,x_1)\,d\xi\,\mu(dy), 
   \end{array}
 \end{equation}
 Now the conclusion follows  by completing the $Q$-divergence introducing $G(t,x)$  defined by \eqref{e3.7},
  and writing
 $$
 \left< \tfrac{Q^{-1} T_\epsilon}{S_\epsilon}\;F(t,x_\xi),Q^{-1} \tfrac{T_{\epsilon\xi}}{S_{\epsilon\xi} }x_\xi) \right>=\mbox{\rm div}_Q\;G(t,x_\epsilon)+\mbox{\rm div}\;G(t,x_\epsilon).
 $$
\end{proof}
 
 \subsection {Bound of the commutator}
It is enough to bound  $B^{2}_\epsilon (u,F)$ because the estimate of   $B^{1}_\epsilon (u,F)$ is analogous to that of  $B^{2,2}_\epsilon (u,F)$. Let us first estimate $\||B^{2,1}_\epsilon (u,F)\|_{L^{p'}(H_T,dt\otimes\mu)}$.
\begin{Proposition}
\label{p3.3}
Let $p>2$, $p'=\tfrac{p}{p-1}$, $s>p'$, $\tfrac{1}{p'}=\tfrac{1}{r}+\tfrac{1}{s}$.
Then we have
\begin{equation}
 \label{e3.15}
 \begin{array}{l}
\ds \|B^{2,1}_\epsilon (u,F)\|_{L^{p'}(H_T,dt\otimes\mu)}\\
\\
\ds\le  C'\|u\|_{L^{r}(H_T,dt\otimes\mu)}  \left(\int_0^T\int_H \mbox{\rm Tr}\;[(DF(x))^{s}]\,dt\,\mu(dx)\right)^{\frac1{s}}. 
\end{array}
\end{equation}
 
 \end{Proposition} 
 \begin{proof}

  We recall that
   $$
  \begin{array}{l}
  \ds B^{2,1}_\epsilon (u,F)(t,x)\\
  \\
  \ds= \frac\epsilon2\int_H\int_0^1 \left[\left< \tfrac{Q^{-1} T_\epsilon}{S_\epsilon}\;DF(t,x_\xi) (Q^{-1} \tfrac{T_{\epsilon\xi}}{S_{\epsilon\xi} }y_\xi),y_\xi\right>+\mbox{\rm div}\;G(t,x_\epsilon)\right]u(t,x_1)\,d\xi\,\mu(dy)\\
  \\
  \ds=: \int_H\int_0^1 H(t,x_\xi,y_\xi) u(t,x_1)\,d\xi\,\mu(dy).
   \end{array}
 $$
 Then
 $$
  \begin{array}{l}
\ds \int_0^T\int_H|B^{2,1}_\epsilon (u,F)(t,x)|^{p'}dt\,\mu(dx)\\
\\
\ds=\int_0^T\int_H\left[ \int_H\int_0^1 H(t,x_\xi,y_\xi) u(t,x_1)\,d\xi\,\mu(dy\right]^{p'}dt\,\mu(dx).
\end{array}
 $$
 Let now $v\in L^p(H_T,dt\otimes\mu)$. Then
  $$
  \begin{array}{l}
\ds  \left|\int_0^T\int_HB^{2,1}_\epsilon (u,F)(t,x)\,v(t,x)dt\,\mu(dx)\right|\\
\\
\ds\le
\int_0^T\int_0^1\int_H\int_H|H(t,x_\xi,y_\xi) u(t,x_1) v(t,x)|d\xi\,\mu(dx)\,dt\,\mu(dy)\\
\\
\ds\le \int_0^1 d\xi \left(\int_0^T\int_H\int_H|H(t,x_\xi,y_\xi) u(t,x_1)|^{p'}\,dt\,\mu(dx)\,\mu(dy)\right)^{1/p'}\|v\|_{L^p(H_T,dt\otimes\mu)}.
     \end{array}
 $$
 By the arbitrariness of $v$ it follows that
 $$
 \|B^{2,1}_\epsilon (u,F)\|_{L^{p'}(H_T,dt\otimes\mu)}  \le  \int_0^1 d\xi \left(\int_0^T\int_H\int_H|H(t,x_\xi,y_\xi) u(x_1)|^{p'}\,dt\,\mu(dx)\,\mu(dy)\right)^{1/p'}.
$$
 Making   the change of variables \eqref{e3.8} and recalling that it  is invariant for $\mu\times \mu$ so that $\mu(dx)\mu(dy)=\mu(dx_\xi)\mu(dy_\xi)$, we get
\begin{equation}
 \label{e3.16d}
 \begin{array}{l}
  \ds \|B^{2,1}_\epsilon (u,F)\|_{L^{p'}(H_T,dt\otimes\mu)} \\
  \\
  \ds \le \int_0^1 d\xi \left(\int_H\int_H|H(t,x,y) |^{s}\,\mu(dx)\,\mu(dy) \right)^{\frac1{s}}\;\|u\|_{L^{r}(H_T,dt\otimes\mu)},
     \end{array}
\end{equation}
 equivalently
 \begin{equation}
 \label{e3.16}
 \begin{array}{l}
   \ds\|B^{2,1}_\epsilon (u,F)\|_{L^{p'}(H_T,dt\otimes\mu)} \le \tfrac\epsilon2\|u\|_{L^{r}(H_T,dt\otimes\mu)}\\
   \\
\ds\times   \int_0^1 d\xi \left(\int_0^T\int_H\int_H|\langle   DG(t,x) (Q^{-1} \tfrac{T_{\epsilon\xi}}{S_{\epsilon\xi} }y ),\tfrac{S_{\epsilon\xi}}{T_{\epsilon\xi} }y \rangle-\mbox{\rm div}\;G(t,x )| |^{s}\,dt\,\mu(dx)\,\mu(dy) \right)^{\frac1{s}}. 
     \end{array}
 \end{equation} 
 Now we can compute explicitly the integral
 \begin{equation}
 \label{e3.17}
 J_1(t,x):=\int_H |\langle   DG(t,x) (Q^{-1} \tfrac{T_{\epsilon\xi}}{S_{\epsilon\xi} }y ),\tfrac{S_{\epsilon\xi}}{T_{\epsilon\xi} }y \rangle-\mbox{\rm div}\;G(t,x )|^{s} \mu(dy),
  \end{equation}
applying  Proposition \ref{p4.3a} from   Appendix B. 
Setting
 $$
 L=\tfrac{S_{\epsilon\xi}}{T_{\epsilon\xi} }DG(t,x) Q^{-1} \tfrac{T_{\epsilon\xi}}{S_{\epsilon\xi} },\quad M=Q^{1/2}LQ^{1/2},
 $$
 we have Tr $M=$ div $G(t,x)$, so that
 $$
 J_1(t,x)=\int_H |\langle   Ly,y \rangle-\mbox{\rm Tr}\;M|^{s} \mu(dy)
 $$
 Then, taking into account \eqref{e4.5a},   we obtain
 \begin{equation}
 \label{e3.18}
 \begin{array}{l}
 \ds J_1(t,x)=C_{s}\;\mbox{\rm Tr}\;[(DG(t,x))^{s}]=C_{s}\mbox{\rm Tr}\;[(Q^{-1}\tfrac{T_\epsilon}{S_\epsilon}\tfrac{T_{\epsilon\xi}}{S_{\epsilon\xi} }DF(t,x))^{s}].
 \end{array}
\end{equation}
Now we estimate  $\|Q^{-1}\tfrac{T_\epsilon}{S_\epsilon}\tfrac{T_{\epsilon\xi}}{S_{\epsilon\xi} }\|$.
 Since
$$
Q^{-1}\tfrac{T_\epsilon}{S_\epsilon}\tfrac{T_{\epsilon\xi}}{S_{\epsilon\xi} }e_k=2\left(\sqrt{\alpha_k}\,\frac{e^{-\alpha_k\epsilon}}{\sqrt{1-e^{-2\alpha_k\epsilon}}}\right)
\left(\sqrt{\alpha_k}\,\frac{e^{-\alpha_k\epsilon}}{\sqrt{1-e^{-2\alpha_k\epsilon}}}\right)e_k.
$$
So
 \begin{equation}
 \label{e3.19}
\|Q^{-1}\tfrac{T_\epsilon}{S_\epsilon}\tfrac{T_{\epsilon\xi}}{S_{\epsilon\xi} }\|\le \frac{C}{\epsilon \;\xi^{1/2}}
\end{equation}
and
$$
J_1(t,x)\le C'_s\frac{C}{\epsilon \;\xi^{1/2}}\;\mbox{\rm Tr}\;[(DF(x))^{s}]
$$
Now by \eqref{e3.16}   we obtain
 \begin{equation}
 \label{e3.20}
 \begin{array}{l}
   \ds\|B^{2,1}_\epsilon (u,F)\|_{L^{p'}(H_T,dt\otimes\mu)}\\
   \\
   \ds\le C'\tfrac12 \|u\|_{L^{r}(H_T,dt\otimes\mu)}    \int_0^1 \xi^{1/2}d\xi\left(\int_0^T\int_H \mbox{\rm Tr}\;[(DF(x))^{s}]\,dt\,\mu(dx)\right)^{\frac1{s}}\\
   \\
   \ds= C'\|u\|_{L^{r}(H_T,dt\otimes\mu)}  \left(\int_0^T\int_H \mbox{\rm Tr}\;[(DF(x))^{s}]\,dt\,\mu(dx)\right)^{\frac1{s}}. 
     \end{array}
 \end{equation} 
 \end{proof}
 
 \begin{Proposition}
 We have
 $$
  \begin{array}{l}
 \ds\|B^{2,2}_\epsilon (u,F)\|_{L^{p'}(H_T,dt\otimes\mu)}\\
 \\
 \ds\le C\left(\int_0^T\int_H(|Q^{-1/2}F(t,x)|^{s}+\mbox{\rm Tr}\;[(DF(t,x))^{s}])\,dt\,\mu(dx)\right)^{1/s}\; \|u\|_{L^r(H_T,dt\otimes\mu)}.
   \end{array}
 $$
 
 \end{Proposition}
 
 \begin{proof}
 
Recall that
$$
B^{2,2}_\epsilon (u,F)(t,x) = \frac\epsilon2\int_H\int_0^1\mbox{\rm div}_Q\;[Q^{-1}\tfrac{T_\epsilon}{S_\epsilon}\tfrac{T_{\epsilon\xi}}{S_{\epsilon\xi} }F(t,x_\xi)]u(t,x_1)\,d\xi\mu(dy).
   $$
   Proceeding as in the proof of \eqref{e3.16}  and  using again the change of variables \eqref{e3.8}, we find
   \begin{equation}
 \label{e3.22c}
 \begin{array}{l}
   \ds\|B^{2,2}_\epsilon (u,F)\|_{L^{p'}(H_T,dt\otimes\mu)}\le \tfrac\epsilon2\|u\|_{L^{r}(H_T,dt\otimes\mu)}\\
   \\
\ds\times   \int_0^1 d\xi \left(\int_0^T\int_H\int_H|\mbox{\rm div}_Q\;[Q^{-1}\tfrac{T_\epsilon}{S_\epsilon}\tfrac{T_{\epsilon\xi}}{S_{\epsilon\xi} }F(t,x)]|^{s}\,dt\,\mu(dx)\,\mu(dy) \right)^{\frac1{s}}. 
     \end{array}
 \end{equation}
 By \eqref{e3.19} we obtain 
   \begin{equation}
 \label{e3.23c}
 \begin{array}{l}
   \ds\|B^{2,2}_\epsilon (u,F)\|_{L^{p'}(H_T,dt\otimes\mu)}\\
   \\
   \ds  \le  C\|u\|_{L^{r}(H_T,dt\otimes\mu)}  \left(\int_0^T\int_H\int_H(\mbox{\rm div}_Q\;[ F(t,x)])^{s}\,dt\,\mu(dx)  \right)^{\frac1{s}}. 
     \end{array}
 \end{equation}
 
 Now  the conclusion follows from Lemma \ref{l5.4} of   Appendix B.

\end{proof}

  \section{Appendix A}
  
  Let $p>1$ be given. Denote by $W_{Q}^{1,p}$ the space of ($\mu$-equivalence
classes of) vector fields $G:H\rightarrow$ $D\left(  Q^{-1/2}\right)  $,
having Fr\'{e}chet differential $DG\left(  x\right)  \in L_{2}\left(
H,H\right)  $ for $\mu$-a.e. $x\in H$, such that
\[
\left\Vert G\right\Vert _{W_{Q}^{1,p}}^{p}:=\int_{H}\left(  \left\vert
Q^{-1/2}G\left(  x\right)  \right\vert ^{p}+\left\Vert DG\left(  x\right)
\right\Vert _{L_{2}\left(  H,H\right)  }^{p}\right)  \mu\left(  dx\right)
<\infty.
\]
The space $W_{Q}^{1,p}$ is a separable Banach space with the norm $\left\Vert
G\right\Vert _{W_{Q}^{1,p}}$. Consider the space $L^{p}\left(  0,T;W_{Q}
^{1,p}\right)  $ with the norm $\left\Vert F\right\Vert _{L^{p}(W_{Q}^{1,p})  }$ defined as
\begin{align*}
\left\Vert F\right\Vert _{L^{p}(  W_{Q}^{1,p})  }^{p} &  =\int
_{0}^{T}\left\Vert F\left(  t,\cdot\right)  \right\Vert _{W_{Q}^{1,p}}^{p}dt\\
&  =\int_{0}^{T}\int_{H}\left(  \left\vert Q^{-1/2}F\left(  t,x\right)
\right\vert ^{p}+\left\Vert DF\left(  t,x\right)  \right\Vert _{L_{2}\left(
H,H\right)  }^{p}\right)  \mu\left(  dx\right)  dt.
\end{align*}

\begin{Lemma}
\label{l4.1a}
Denote by $\mathcal{V}_{p}$ the family of all functions $F_{n}\in L^{p}\left(
0,T;W_{Q}^{1,p}\right)  $ of the form
\[
F_{n}\left(  t,x\right)  =\sum_{i=1}^{n}\varphi_{i}^{n}\left(  t,\left\langle
x,e_{1}\right\rangle ,...,\left\langle x,e_{n}\right\rangle \right)  e_{i}
\]
with $\varphi_{i}^{n}\in C_{0}^{1}\left(  \left[  0,T\right]  \times
\mathbb{R}^{n},\mathbb{R}\right)  $. Then $\mathcal{V}_{p}$ is dense in
$L^{p}\left(  0,T;W_{Q}^{1,p}\right)  $.
\end{Lemma}

\begin{proof}
We proceed by a sequence of reductions of the problem:\ from general elements
$F$ of $L^{p}\left(  0,T;W_{Q}^{1,p}\right)  $ to piece-wise constant (in
time) functions;\ then finitely based; then also with values in finite
dimensional spaces; and finally smooth.

\textbf{Step 1}. Let $\mathcal{V}_{p}^{1}$ be the family of all piece-wise
constant functions $F:\left[  0,T\right]  \rightarrow W_{Q}^{1,p}$, namely of
the form
\[
F\left(  t,\cdot\right)  =\sum_{i=1}^{k-1}F_{i}1_{\left[  t_{i},t_{i+1}
\right]  }\left(  t\right)
\]
where $0\leq t_{1}\leq...\leq t_{k}\leq T$, and $F_{i}\in W_{Q}^{1,p}$. It is
a known fact that $\mathcal{V}_{p}^{1}$ is dense in $L^{p}\left(
0,T;W_{Q}^{1,p}\right)  $. Thus, to prove the lemma, it is sufficient to prove
that any element $F\in\mathcal{V}_{p}^{1}$ can be approximated by a sequence
$\left\{  F_{n}\right\}  \subset\mathcal{V}_{p}$, in the sense of
$\lim_{n\rightarrow\infty}\left\Vert F_{n}-F\right\Vert _{L^{p}\left(
W_{Q}^{1,p}\right)  }^{p}=0$.

\textbf{Step 2}. Any $G\in W_{Q}^{1,p}$ is the limit in $\left\Vert
\cdot\right\Vert _{W_{Q}^{1,p}}$ of a sequence $\left\{  G_{n}\right\}
\subset W_{Q}^{1,p}$ having the following property: $G_{n}\left(  x\right)
=G_{n}\left(  \pi_{n}x\right)  $ (namely they are \textit{finitely based}),
where $\pi_{n}x=\sum_{i=1}^{n}\left\langle x,e_{i}\right\rangle e_{i}$.
Indeed, define
\begin{align*}
G_{n}\left(  x\right)   &  :=\int_{H}G\left(  \pi_{n}x+\left(  1-\pi
_{n}\right)  y\right)  \mu\left(  dy\right)  \\
H_{n}\left(  x\right)   &  :=Q^{-1/2}G_{n}\left(  x\right)  =\int_{H}
Q^{-1/2}G\left(  \pi_{n}x+\left(  1-\pi_{n}\right)  y\right)  \mu\left(
dy\right)
\end{align*}
In \cite[Corollary 3.5.2]{Bo98}  it is proved that $H_{n}\rightarrow$
$Q^{-1/2}G$ in $L^{2}\left(  H,\mu\right)  $, which is the first part of the
property $\left\Vert G_{n}-G\right\Vert _{W_{Q}^{1,p}}\rightarrow0$. The
second one is proved in \cite[Proposition 5.4.5]{Bo98}.

\textbf{Step 3}. Any $G\in W_{Q}^{1,p}$ is the limit in $\left\Vert
\cdot\right\Vert _{W_{Q}^{1,p}}$ of a sequence $\left\{  G_{n}\right\}
\subset W_{Q}^{1,p}$ having the following property: $G_{n}\left(  x\right)
=\pi_{n}G_{n}\left(  \pi_{n}x\right)  $ (namely they are finitely based and
have values in a finite dimensional space). The proof (using Step 2) is
elementary. 

From these facts it follows that any element $F\in\mathcal{V}_{p}^{1}$,
$F\left(  t,\cdot\right)  =\sum_{i=1}^{k-1}F_{i}1_{\left[  t_{i}
,t_{i+1}\right]  }\left(  t\right)  $, can be approximated in $L^{p}\left(
W_{Q}^{1,p}\right)  $-norm by $F_{n}\in\mathcal{V}_{p}^{1}$ of the form
\[
F_{n}\left(  t,\cdot\right)  =\sum_{i=1}^{k-1}F_{i}^{n}1_{\left[
t_{i},t_{i+1}\right]  }\left(  t\right)
\]
where each $F_{i}^{n}$ has the property $F_{i}^{n}\left(  x\right)  =\pi
_{n}F_{i}^{n}\left(  \pi_{n}x\right)  $ and $F_{i}^{n}$ converges to $F_{i}$
in $W_{Q}^{1,p}$.

In other words, we have proved that any $F\in\mathcal{V}_{p}^{1}$ is the limit
in $L^{p}\left(  W_{Q}^{1,p}\right)  $-norm of a sequence $F_{n}$ of the form
\[
F_{n}\left(  t,x\right)  =\sum_{i=1}^{n}\varphi_{i}^{n}\left(  t,\left\langle
x,e_{1}\right\rangle ,...,\left\langle x,e_{n}\right\rangle \right)  e_{i}
\]
where $\varphi_{i}^{n}$ are piece-wise constant in $t$ and of class
$W^{1,p}\left(  \mathbb{R}^{n},\gamma_{n}\right)  $ in space, where
$\gamma_{n}$ is the centered symmetric Gaussian measure on $\mathbb{R}^{n}$
($\gamma_{n}$ is equivalent to the Gaussian measure on $\mathbb{R}^{n}$
corrsponding to the projection of $\mu$ by $\pi_{n}$, and the spaces $W^{1,p}$ coincide).

\textbf{Step 4}. Any element $\varphi_{i}^{n}$ of class $L^{2}\left(
0,T;W^{1,p}\left(  \mathbb{R}^{n},\gamma_{n}\right)  \right)  $ is limit, in
such topology, of $C_{0}^{1}\left(  \left[  0,T\right]  \times\mathbb{R}
^{n},\mathbb{R}\right)  $-functions. The proof is complete.
\end{proof}
 
  \section{Appendix B}
  
 \subsection{Computation of some integrals}
 Let $\mu=N_Q$ and assume that the sequence $(\lambda_k)$ of eigenvalues of $Q$ be nondecreasing. 
\begin{Lemma}
\label{l4.1}
Assume that  $L\in L(H)$ is symmetric and compact. Then there is  $\epsilon_0>0$ such that
 \begin{equation}
\label{e4.2}
\int_He^{-\epsilon\langle Lx,x\rangle}N_{Q}(dx)= 
[\det(1+2\epsilon Q^{1/2}LQ^{1/2})]^{-1/2},\quad\mbox{\rm if}\; \epsilon<\epsilon_0.
\end{equation}
($\epsilon_0$  is determined by the condition  $1+2\epsilon_0\mu>0$ where $\mu$ are eigenvalues of  $ Q^{1/2}LQ^{1/2})$)
\end{Lemma}
\begin{proof}
  Set  $M=Q^{1/2}LQ^{1/2}$, $M$ is obviously compact. Let $(f_k)$ be  an orthonormal  basis of eigenvectors  of  $M$ and   $(\beta_k)$ the corresponding sequence of eigenvalues. Then we have
  \begin{equation}
\label{e4.3}
\langle Lx,x  \rangle =\langle MQ^{-1/2}x,Q^{-1/2}x  \rangle=\sum_{k=1}^n\beta_k |\langle Q^{-1/2}x,f_k  \rangle|^2,
\end{equation}
 so that
 $$
\int_He^{-\epsilon  \langle Lx,x  \rangle}N_Q(dx)=\int_He^{-\epsilon  \sum_{k=1}^\infty \beta_k |\langle Q^{-1/2}x,f_k  \rangle|^2}N_Q(dx) 
 $$
 Since $(f_k)$  is an  orthogonal system,  the sequence of real random variables     $x\to\langle Q^{-1/2}x,f_k  \rangle,\;k\in\N$ (whose law is $N_1$) are independent.  Consequently
 $$
 \begin{array}{l}
 \ds\int_He^{-\epsilon  \langle Lx,x  \rangle}N_Q(dx)= \prod_{k=1}^n \int_He^{-\epsilon \beta_k |\langle Q^{-1/2}x,f_k  \rangle|^2}N_{Q}(dx)\\
 \\
 \ds=\prod_{k=1}^n(1+2\epsilon \beta_k )^{-1/2}=[\det(1+2\epsilon Q^{1/2}LQ^{1/2}]^{-1/2}
 \end{array},
 $$
 as claimed.
   \end{proof}
   \begin{Remark}
\label{r4.2}
\em If $L$ is compact but not symmetric we have
$$
\int_He^{-\epsilon\langle Lx,x\rangle}N_{Q}(dx)= 
[\det(1+2\epsilon Q^{1/2}L_sQ^{1/2})]^{-1/2},\quad\mbox{\rm if}\; \epsilon<\epsilon_0,
$$
where
$$
L_s=\frac12\;(L+L^*)
$$
and $L^*$ is the adjoint of $L$.

\end{Remark}
   Set now
   \begin{equation}
\label{e4.4}
\begin{array}{lll}
S(\epsilon)&=&\ds\int_He^{-\epsilon ( \langle Lx,x  \rangle-\tiny\mbox{Tr}\;[Q^{1/2}LQ^{1/2}]}N_Q(dx)\\
\\
&=&[\det(1+2\epsilon Q^{1/2}LQ^{1/2}]^{-1/2}e^{\epsilon  \tiny\mbox{Tr}\;[Q^{1/2}LQ^{1/2}]}.
\end{array}
\end{equation}
Notice that $S(0)=1$.
Then for any $m\in\N$ we have
  \begin{equation}
\label{e4.5}
\int_H   (\langle Lx,x\rangle-\mbox{Tr}\;[Q^{1/2}LQ^{1/2}])^m N_Q(dx)=(-1)^m S^{(m)}(0).
\end{equation}
\begin{Proposition}
\label{p4.3a}
For any $m\in\N$  there is $C_m>0$ such that.
  \begin{equation}
\label{e4.5a}
\int_H   (\langle Lx,x\rangle-\mbox{\rm Tr}\;[Q^{1/2}LQ^{1/2}])^m N_Q(dx)=C_m \mbox{\rm Tr}\;[(Q^{1/2}LQ^{1/2})^m]).
\end{equation}
\end{Proposition}
\begin{proof}
Setting  $M=Q^{1/2}LQ^{1/2}$ we have
\begin{equation}
\label{e4.6}
\begin{array}{lll}
   S'(\epsilon)&=&\mbox{\rm Tr}\;[M-M(1+2\epsilon M)^{-1}]S(\epsilon)\\
   \\
  & =&2\epsilon\mbox{\rm Tr}\;[M^2(1+2\epsilon M)^{-1}]S(\epsilon).
   \end{array}
\end{equation}
In particular, $S'(0)=0$.

Now set
$$
F(\epsilon)=\log S(\epsilon).
$$
Then
$$
F'(\epsilon)=2\epsilon\mbox{\rm Tr}\;[M^2(1+2\epsilon M)^{-1}]
$$
and
$$
F^{(n)}(\epsilon)=(-1)^{n+1} 2^{n-1}(n-1)!\;\mbox\;{\rm Tr}\;[M^n(1+2\epsilon M)^{-n}],\quad n\ge 2.
$$
Therefore
$$
F'(0)=0,\quad F^{(n)}(0)=(-1)^{n+1} 2^{n-1}(n-1)!\;\mbox\;{\rm Tr}\;[M^n]=:k_n\;\mbox{\rm Tr}\;[M^n],\quad n\ge 2.
$$
Now $S(\epsilon)=e^{F(\epsilon)}$ and it is easy to see by recurrence that there is $C_n>0$  such  that 
$$
S^{(n)}(0)\le C_n\mbox{\rm Tr}\;[M^n].
$$
The conclusion follows.
\end{proof}

\subsection{An estimate for Gaussian divergences}   

In the next lemma, $G$ is a vector field of the form
\[
G\left(  x\right)  =\sum_{i=1}^{n}\varphi_{i}\left(  \left\langle
x,e_{1}\right\rangle ,...,\left\langle x,e_{n}\right\rangle \right)  e_{i}
\]
with $\varphi_{i}\in C_{0}^{1}\left(  \mathbb{R}^{n},\mathbb{R}\right)  $,
where $\left\{  e_{i}\right\}  $ is a c.o.s. of $H$ of eigenvectors of $Q$.
For them we may define $\operatorname{div}_{Q}G\left(  y\right)  ={\rm Tr}\left(
DG\left(  y\right)  \right)  -\left\langle y,Q^{-1}G\left(  y\right)
\right\rangle $.

\begin{Lemma}
\label{l5.4}
For every $p>1$ there is a constant $C_{p}>0$ such that
\[
\int_{H}\left\vert \operatorname{div}_{Q}G\left(  y\right)  \right\vert
^{p}\mu\left(  dy\right)  \leq C_{p}\int_{H}\left(  \left\Vert DG\left(
y\right)  \right\Vert _{L_{2}\left(  H,H\right)  }^{2}+\left\vert
Q^{-1/2}G\left(  y\right)  \right\vert ^{p}\right)  \mu\left(  dy\right)
\]
for every vector field $G$ as above.
\end{Lemma}

\begin{proof}
The following result is classical in Malliavin calculus (see for instance
\cite[Proposition 5.8.8]{Bo98}): for every $p>1$ there is a constant $C_{p}>0$
such that for every $n\in\mathbb{N}$, if $\gamma_{n}$ denotes the symmetric
centered Gaussian measure in $\mathbb{R}^{n}$, then
\[\begin{array}{l}
\ds\int_{\mathbb{R}^{n}}\left\vert {\rm Tr}\left(  DF\left(  \widetilde{x}\right)
\right)  -\left\langle \widetilde{x},F\left(  \widetilde{x}\right)
\right\rangle \right\vert ^{p}\gamma_{n}\left(  d\widetilde{x}\right) \\
\\
\ds \leq
C_{p}\int_{\mathbb{R}^{n}}\left(  \left\Vert DF\left(  \widetilde{x}\right)
\right\Vert _{L_{2}\left(  \mathbb{R}^{n},\mathbb{R}^{n}\right)  }
^{2}+\left\vert F\left(  \widetilde{x}\right)  \right\vert _{\mathbb{R}^{n}
}^{p}\right)  \gamma_{n}\left(  d\widetilde{x}\right)
\end{array}
\]
for all smooth compact support vector fields $F:\mathbb{R}^{n}\rightarrow
\mathbb{R}^{n}$.

Set $H_{n}=\pi_{n}\left(  H\right)  $, $\pi_{n}x=\sum_{i=1}^{n}\left\langle
x,e_{i}\right\rangle e_{i}$ and let $J:H_{n}\rightarrow\mathbb{R}^{n}$ be the
natural isomorphism. The operators $Q$, $Q^{1/2}$, $Q^{-1/2}$ work as
operators $H_{n}$, hence they\ define corresponding operators $Q_{n}$,
$Q_{n}^{1/2}$, $Q_{n}^{-1/2}$ in $\mathbb{R}^{n}$.

Given $G$ as above, consider the vector field $F:\mathbb{R}^{n}\rightarrow
\mathbb{R}^{n}$ defined as
\[
F\left(  \widetilde{x}\right)  :=Q_{n}^{-1/2}JG\left(  J^{-1}Q_{n}%
^{1/2}\widetilde{x}\right)  .
\]
With little abuse of notations, it is simply the map $F\left(  x\right)
:=Q^{-1/2}G\left(  Q^{1/2}x\right)  $. We have%
\begin{align*}
DF\left(  \widetilde{x}\right)    & =J\left(  DG\right)  \left(  J^{-1}%
Q_{n}^{1/2}\widetilde{x}\right)  \\
{\mbox Tr}\left(  DF\left(  \widetilde{x}\right)  \right)  -\left\langle
\widetilde{x},F\left(  \widetilde{x}\right)  \right\rangle  & ={\mbox Tr}\left(
J\left(  DG\right)  \left(  J^{-1}Q_{n}^{1/2}\widetilde{x}\right)  \right)
-\left\langle \widetilde{x},Q_{n}^{-1/2}JG\left(  J^{-1}Q_{n}^{1/2}%
\widetilde{x}\right)  \right\rangle \\
& =\operatorname{div}_{Q}G\left(  y\right)  |_{y=J^{-1}Q_{n}^{1/2}%
\widetilde{x}}%
\end{align*}
hence we have%
\begin{align*}
& \int_{\mathbb{R}^{n}}\left\vert \operatorname{div}_{Q}G\left(  y\right)
|_{y=J^{-1}Q_{n}^{1/2}\widetilde{x}}\right\vert ^{p}\gamma_{n}\left(
d\widetilde{x}\right)  \\
& \leq C_{p}\int_{\mathbb{R}^{n}}\left(  \left\Vert J\left(  DG\right)
\left(  J^{-1}Q_{n}^{1/2}\widetilde{x}\right)  \right\Vert _{L_{2}\left(
\mathbb{R}^{n},\mathbb{R}^{n}\right)  }^{2}+\left\vert Q_{n}^{-1/2}JG\left(
J^{-1}Q_{n}^{1/2}\widetilde{x}\right)  \right\vert _{\mathbb{R}^{n}}%
^{p}\right)  \gamma_{n}\left(  d\widetilde{x}\right)  .
\end{align*}
If we denote by $\mu_{n}$ the image measure, on $H_{n}$, of $\gamma_{n}$ under
the transformation $\widetilde{x}\mapsto y=J^{-1}Q_{n}^{1/2}\widetilde{x}$, we
have proved
\[
\int_{H}\left\vert \operatorname{div}_{Q}G\left(  y\right)  |\right\vert
^{p}\mu_{n}\left(  dy\right)  \leq C_{p}\int_{H}\left(  \left\Vert J\left(
DG\right)  \left(  y\right)  \right\Vert _{L_{2}\left(  \mathbb{R}%
^{n},\mathbb{R}^{n}\right)  }^{2}+\left\vert Q_{n}^{-1/2}JG\left(  y\right)
\right\vert _{\mathbb{R}^{n}}^{p}\right)  \mu_{n}\left(  dy\right)  .
\]
It is now easy to realize that this is the claim of the lemma, taking into
account the special form of $G$. The proof is complete.
\end{proof}

\end {document}